\documentclass[final,3p,times]{elsarticle}

\usepackage{amssymb}
\usepackage{amsthm}

\usepackage{amssymb}
\usepackage[]{amsmath}
\usepackage{graphics}
\usepackage{amssymb}
\usepackage[]{amsmath}
\usepackage{amsthm}
\usepackage{setspace}
\usepackage{epsfig}
\usepackage{subfigure}
\usepackage{booktabs}
\usepackage{float}
\usepackage{graphicx}
\makeatletter

\def\ps@pprintTitle{%
     \let\@oddhead\@empty
     \let\@evenhead\@empty
     \def\@oddfoot{\reset@font\hfil\thepage\hfil}
     \let\@evenfoot\@oddfoot}

\newcommand{\Rmnum}[1]{\expandafter\@slowromancap\romannumeral #1@}
\makeatother
 \newtheorem{lemma}{Lemma}[section]
  \newtheorem{theorem}{Theorem}[section]
  \newtheorem{proposition}{Proposition}[section]
    \newtheorem{remark}{Remark}[section]
\biboptions{numbers,sort&compress}
\usepackage[dvipdfm,colorlinks,linkcolor=red,anchorcolor=blue,citecolor=green]{hyperref} %
\usepackage{amsmath}
\usepackage{anysize}
\marginsize{2.5cm}{2.5cm}{1cm}{1cm}

\selectfont

\begin{document}

\begin{frontmatter}
\title{ Monotonicity  of the periodic waves for the perturbed generalized defocusing mKdV equation}

\author{Lin Lu}
\author{Aiyong Chen\corref{mycorrespondingauthor}}
\cortext[mycorrespondingauthor]{Corresponding author.}
\ead{aiyongchen@163.com}
\author{Xiaokai He}
\address{School of Mathematics and Statistics, Hunan First Normal University,
Changsha, 410205,  P R China}

%
%

\begin{abstract}
In this paper, we study the existence of periodic waves for the perturbed  generalized defocusing mKdV equation
 using the theory of geometric singular perturbation. 
 By Abelian integral and involution operation, we prove that the limit wave speed
  $c_0(h)$ is monotonic with respect to energy $h$,
 and  the lower bound of the limit wave speed is found. 
 These works extend the main result of Chen et al. (2018) to the generalized  case. 
Some numerical simulations are conducted to verify the correctness of the theoretical analysis.
\end{abstract}
 \begin{keyword}
 Periodic wave; Monotonicity; perturbed analysis;  Defocusing mKdV equation.
  \vspace{0.8em}
  \MSC[2020] 34C25;  34C60; 37C27
 \end{keyword}
\makeatletter
\@namedef{subjclassname@1991}{Subject}
\@namedef{subjclassname@2000}{Subject}
\@namedef{subjclassname@2010}{Subject}
\makeatother
\end{frontmatter}

\section{Introduction}\label{s1}

The focusing and defocusing KdV equations are two different forms of nonlinear wave equations, which have significant differences 
in physical background and solution properties.
The focusing KdV equation can be represented by
\begin{equation}\label{F-kdv}
 u_t+u u_x+u_{xxx}=0.
 \end{equation}
Equation  (\ref{F-kdv})  can be used to describe surface waves or unstable wave modes, which may lead to explosive growth of waves.
  The solutions to  equation (\ref{F-kdv}) usually exhibit wave focusing phenomena, and 
the sharpening of waves may lead to the formation of  singular solutions. 
The defocusing KdV equation is
\begin{equation}\label{deF-kdv}
 u_t-u u_x+u_{xxx}=0,
  \end{equation}
   which is mainly used to study soliton propagation and stable waves. 
The solutions to equation  (\ref{deF-kdv}) often exhibit soliton waves, which have relatively stable properties.
 The gradual dispersion of waves during diffusion generally cannot lead to solution explosion.
To further investigate the influence of nonlinear term  on the dynamic behavior of  equations  (\ref{F-kdv})  and  (\ref{deF-kdv}),
 many scholars have modified the nonlinear term of the above two equations,
  and referred to the modified equations as the  focusing  mKdV equation and the  defocusing mKdV equation, respectively.

  The  focusing and defocusing mKdV equations have more types of solutions, such as soliton solutions, 
  multiple soliton solutions and periodic solutions and so on.
In recent years, many scholars have studied the dynamic behavior of the focusing and defocusing mKdV equations. For instance, 
Christ et al. \cite{Christ2003} found some modified scattering solutions
 that exist globally in time for the defocusing  mKdV  equation.
Zhang and Yan \cite{ZhangG2020,ZhangG2020410} have studied the inverse scattering transformation and soliton solutions and 
 soliton interaction of the focusing and defocusing mKdV equations under non-zero boundary conditions.
Chen and Fan \cite{ChenM2020} discussed the long-time asymptotic behavior of the solution for the discrete defocusing mKdV equation 
 by using the  nonlinear steepest descent method.
Wang et al. \cite{WangDS2022} investigated 
the complete classification of solutions to the defocusing complex mKdV equation.
Mucalica and Pelinovsky \cite{Mucalica2024} 
 constructed a new exact solution for the interaction between dark solitons and periodic waves in the defocusing KdV equation.
Xu et al. \cite{XuT2003} studied the long-term asymptotic behavior in soliton free regions for the defocusing mKdV equation 
with finite density initial data.

Many scholars have conducted relevant research on the existence of  periodic solutions for the perturbed  focusing mKdV equation.
In 2014, Yan et al. \cite{YanW2014} investigated the existence of solitary waves and periodic waves to 
the perturbed generalized focusing mKdV equation
\begin{equation}\label{focusing-eq1}
U_t+U^n U_x +U_{xxx} +\epsilon (U_{xx} + U_{xxxx}) =0,
\end{equation}
where $n>0$ is an integer  and $\epsilon>0$ is a small parameter. 
When $n=1,$ Ogawa \cite{Ogawa1994} proved that the limit wave speed $c_0 (h)$ of periodic 
waves of equation (\ref{focusing-eq1}) is monotonic with respect to $h$.
When $n=2,3,4,$ some similar results are found by  the authors \cite{Chen2023,Chen2021}. 
Recently, Patra and Rao \cite{Patra2024} studied the monotonicity of limit wave speed $c_0(h)$ of
 periodic  solutions of equation (\ref{focusing-eq1}) for any positive integer $n$.  
In 2018, Chen et al. \cite{ChenA2018}  considered the existence of kink waves and periodic waves of 
  the perturbed  defocusing mKdV equation
\begin{equation*}
U_t-U^2 U_x +U_{xxx} +\epsilon (U_{xx} + U_{xxxx}) =0,
\end{equation*}
where $\epsilon>0$ is a small parameter.

 Motivated by the above works, in the present paper we consider 
the perturbed  generalized  defocusing mKdV equation
\begin{equation}\label{1}
U_t-U^n U_x +U_{xxx} +\epsilon (U_{xx} + U_{xxxx}) =0,
\end{equation}
where $n>0$ is an integer and $\epsilon>0$ is a small parameter.  
We study the monotonicity of the limit wave speed  of  periodic traveling waves of equation (\ref{1}).
The stability of periodic waves is closely related to the monotonicity of period function and  limit wave speed,
and the studies of the monotonicity of period function and limit wave speed have been widely used
 in various fields such as fluid dynamics, physics, and optics
  \cite{Chen2018,Chen2016,LuL2024,LuL2023,Geyer2022,Johnson2014,WeiM2024,FanF2024}.
 Our works of this paper can extend the main result of \cite{ChenA2018} to the generalized case.

The rest of this paper is arranged as follows. In Section \ref{s2},  we discuss the phase portraits of
the unperturbed system of equation (\ref{1}) and we give the main result of this paper. 
In Section  \ref{s3}, the perturbation analysis is used to study the existence of periodic solutions of 
equation  (\ref{1}).
 In Section \ref{szeng3}, 
  we complete the proof of the main result  of Section \ref{s2} by using  Abelian integral and involution.

\section{ Main result  }  \label{s2}

Let
\begin{equation}\label{2}
 U(x,t)=U(\xi),~~~~~~~~~~~\xi=x+ c t,
\end{equation}
Where $c>0$ is a parameter.
Substituting (\ref{2}) into equation (\ref{1}), we have
\begin{equation}\label{3}
 c \frac{\mathrm{d} U}{\mathrm{d} \xi} - U^n  \frac{\mathrm{d} U}{\mathrm{d} \xi}  + \frac{\mathrm{d}^3 U}{\mathrm{d} \xi^3} +
  \epsilon \left(  \frac{\mathrm{d}^2 U}{\mathrm{d} \xi^2}   +  \frac{\mathrm{d}^4 U}{\mathrm{d} \xi^4}   \right)=0.
\end{equation}
 Integrating equation (\ref{3}) and
null the integration constant, we obtain
\begin{equation}\label{4}
 c U - \frac{1}{n+1} U^{n+1}  + \frac{\mathrm{d}^2 U}{\mathrm{d} \xi^2} +
  \epsilon \left(  \frac{\mathrm{d} U}{\mathrm{d} \xi}   +  \frac{\mathrm{d}^3 U}{\mathrm{d} \xi^3}   \right)=0.
\end{equation}
Let 
\begin{equation}\label{5}
 U=\sqrt[n]{c} u, ~~~~~~~~~~ \xi= \frac{\tau}{\sqrt{c}}.
\end{equation}
Then  equation (\ref{4}) becomes 
\begin{equation}\label{6}
 u - \frac{1}{n+1} u^{n+1}  + \frac{\mathrm{d}^2 u}{\mathrm{d} \tau^2} +
  \epsilon \left(  \frac{\mathrm{d} u}{\mathrm{d} \tau}  \frac{1}{\sqrt{c}}   + \sqrt{c} \frac{\mathrm{d}^3 u}{\mathrm{d} \tau^3}   \right)=0.
\end{equation}
If $\epsilon = 0$, then equation  (\ref{6}) is an unperturbed Hamiltonian system
\begin{equation}\label{8}
 \left\{\begin{split}
&\dfrac{\mathrm{d}u}{\mathrm{d} \tau}=v,\\
&\dfrac{\mathrm{d}v}{\mathrm{d} \tau}=-u + \frac{1}{n+1} u^{n+1},\\
\end{split}\right.
\end{equation}
with Hamiltonian
\begin{equation}\label{9}
 H(u,v)=\frac{v^2}{2} + W(u),
\end{equation}
where $W(u)=\frac{  u^2  }{2} - \frac{ u^{n+2} }{(n+1)(n+2)}.$

When $n>0$ is an even number, system (\ref{8}) has three equilibrium points $O(0,0)$ and $E_{1 \pm} (\pm \sqrt[n]{n+1},0).$
 By the theory of planar dynamical systems \cite{Chow1981,LiJB2013}, we obtain that $O(0,0)$ is a center and $E_{1 \pm} (\pm \sqrt[n]{n+1},0)$ are two saddles. 
Similarly, when $n>0$ is an odd number, system (\ref{8}) has two equilibrium points $O(0,0)$ and $E_{1 +} (\sqrt[n]{n+1},0),$
 where $O(0,0)$ is a center and $E_{1 +} ( \sqrt[n]{n+1},0)$ is a saddle. 
  The phase portraits of system (\ref{8}) are plotted in Figure  \ref{Fig1}.
  Therefore, for any positive integer $n$, the equilibrium point $O(0,0)$ 
is a center of the unperturbed  system  (\ref{8}).
 Note that $H(0,0)=0$ and $H(\pm \sqrt[n]{n+1},0)=\frac{n(n+1)^{ \frac{2}{n} }   }{2(n+2)}.$
 When $H(u,v)=h$ and $h \in (0,\frac{n(n+1)^{ \frac{2}{n} }   }{2(n+2)}),$ 
 there exist a set of periodic orbits surrounding the center $O(0,0).$
  We want to discuss the existence of periodic solutions  of the perturbed equation   (\ref{6}). 
 In the following,  we give the main result of this paper.
 \begin{figure}[H]
\centering
\subfigure[ When $n$ is an even number.]{
\includegraphics[width=2.3in]{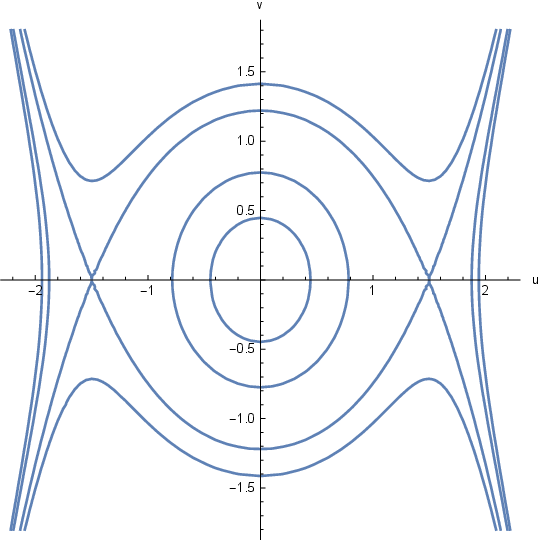}}
\hspace{0.2in}
\subfigure[When $n$ is an odd number.]{
\includegraphics[width=2.3in]{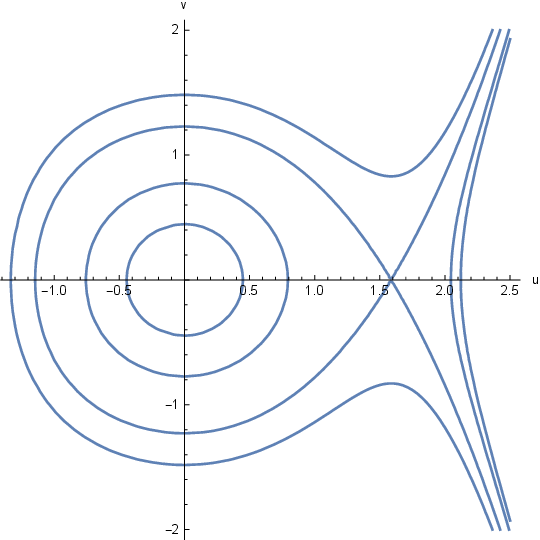}}
\caption{  The phase portraits of system (\ref{8}) \label{Fig1}.}
\end{figure}

 \begin{theorem}\label{theorem-0}
 For any  positive integer $n$, let $$d_n= \frac{n(n+1)^{ \frac{2}{n} }   }{2(n+2)}.$$ Then
 there exists $\epsilon_n^*>0$ such that for each $\epsilon \in (0,\epsilon_n^*)$
and $h\in(0,d_n),$ equation (\ref{1}) has a traveling wave solution 
\begin{equation}
U =  \sqrt[n]{c} u(\epsilon,h,c,\tau),
\end{equation}
where  $c=c(\epsilon,h)$, and $u(\epsilon,h,c,\tau)$ is
a solution of equation  (\ref{6}). 
Furthermore, let
 $c_0(h)=\lim\limits_{\epsilon\rightarrow 0}c(\epsilon,h),$
then $c_0(h)$ satisfies $c_0^\prime(h)>0$  for  $ h \in (0,d_n) $
and $\lim\limits_{h \rightarrow 0} c_0(h)=1.$
\end{theorem}

\section{Perturbation analysis}\label{s3}

We first introduce a technical lemma.

\begin{lemma}[\cite{YanW2014}]      \label{Le1}
Consider the system
\begin{equation}\label{10}
 \left\{\begin{split}
&  \dot{x} = f(x,y,\epsilon),\\
&   \dot{y} = \epsilon g(x,y,\epsilon),\\
\end{split}\right.
\end{equation}
 where $x \in \mathbf{R}^n$, $y  \in \mathbf{R}^l$ and $\epsilon$ is a real parameter,
  $f$ and $g$ are $C^\infty$ on the set $V \times Y$,
where $V \in \mathbf{R}^{n+l}$ and $Y$ is an open interval containing zero. Assume that for $\epsilon = 0$, the
system has a compact normally hyperbolic manifold $M_0$ which is contained in the set
$f (x, y, 0) = 0$. The manifold $M_0$ is said to be normally hyperbolic if the linearization
of (\ref{10}) at each point in $M_0$ has exactly $\mathrm{dim}(M_0)$ eigenvalues on the imaginary axis
 $\mathfrak{R}(\lambda) = 0$. Then for any
$0 < r < +\infty$, if $\epsilon > 0$ sufficiently small, there exists a manifold $M_\epsilon$
such that the following conclusions hold.

(1) $M_\epsilon$ is locally invariant under the flow of (\ref{10}).

(2)  $M_\epsilon$ is $C^r$ in $x, y$ and $\epsilon$.

(3) $M_\epsilon = \{(x, y) | x = h^ \epsilon(y)\}$ for some $C^r$ function $h^\epsilon$ and $y$ in some compact set $K$.

(4) There exist locally invariant stable and unstable manifolds $W^s(M_\epsilon)$ and $W^u(M_\epsilon)$
which are  $O(\epsilon)$ close and  diffeomorphic to $W^s(M_0)$ and $W^u(M_0)$, respectively.
\end{lemma}

Obviously, the perturbed equation   (\ref{6})  is equivalent to the following dynamical system
 \begin{equation}\label{7}
 \left\{\begin{split}
&\dfrac{\mathrm{d}u}{\mathrm{d} \tau}=v,\\
&\dfrac{\mathrm{d}v}{\mathrm{d} \tau}=w,\\
& \epsilon \sqrt{c} \dfrac{\mathrm{d}w}{\mathrm{d} \tau}=-u + \frac{1}{n+1} u^{n+1} -w -\frac{\epsilon}{\sqrt{c}} v.
\end{split}\right.
\end{equation}
When $\epsilon>0,$ the solutions of system (\ref{7}) evolve in the three-dimensional 
$(u, v, w)$ phase space, and system (\ref{7}) has three
equilibria  $(0,0,0)$ and $( \pm \sqrt[n]{n+1},0,0)$ in this phase space.
Taking the transformation $\tau= \epsilon \lambda,$ system (\ref{7})  yields
 \begin{equation}\label{11}
 \left\{\begin{split}
&\dfrac{\mathrm{d}u}{\mathrm{d} \lambda}= \epsilon v,\\
&\dfrac{\mathrm{d}v}{\mathrm{d} \lambda}= \epsilon w,\\
&  \sqrt{c} \dfrac{\mathrm{d}w}{\mathrm{d} \lambda}=-u + \frac{1}{n+1} u^{n+1} -w -\frac{\epsilon}{\sqrt{c}} v.
\end{split}\right.
\end{equation}
We can see that  system (\ref{7})  is  equivalent to  system (\ref{11}) for $\epsilon > 0$.

When $\epsilon\rightarrow0,$ there exists a critical manifold 
\begin{equation*}
M_0 =  \left\{ (u,v,w)\in \mathbf{R}^3  \Big|  w= -u + \frac{1}{n+1} u^{n+1}         \right \}  
\end{equation*}
for system (\ref{7}).
When $\epsilon\rightarrow0,$ system (\ref{11}) goes to 
\begin{equation}\label{12}
 \left\{\begin{split}
&\dfrac{\mathrm{d}u}{\mathrm{d} \lambda}= 0,\\
&\dfrac{\mathrm{d}v}{\mathrm{d} \lambda}= 0,\\
&  \sqrt{c} \dfrac{\mathrm{d}w}{\mathrm{d} \lambda}=-u + \frac{1}{n+1} u^{n+1} -w.
\end{split}\right.
\end{equation}
The linearization of  (\ref{12})  is given by the matrix 
\begin{equation*}
\left(
\begin{array}{ccc}
0 & 0  &  0 \\
0 & 0  &  0 \\
\frac{1}{\sqrt{c}}(-1+u^n)   &   0   & -\frac{1}{\sqrt{c}}\\
\end{array} \right ).
\end{equation*}
Since the eigenvalues of the above matrix are $0, 0$ and $-\frac{1}{\sqrt{c}},$ 
it follows that $M_0$ is normally hyperbolic. From Lemma \ref{Le1}, we obtain that for $\epsilon>0$ sufficiently small,
there exists a two-dimensional submanifold $M_\epsilon$ which is invariant under the flow of  (\ref{7}). Moreover, $M_\epsilon$
can be written as
\begin{equation*}
M_\epsilon =  \left\{ (u,v,w)\in \mathbf{R}^3  \Big|  w= -u + \frac{1}{n+1} u^{n+1}  + g_1(u,v,\epsilon)       \right \} , 
\end{equation*}
where $g_1(u,v,\epsilon)$ is a smooth function with respect to $\epsilon$ and satisfies  $g_1(u,v,0)=0.$
By using Taylor series
$$ g_1(u,v,\epsilon)=\epsilon g_2(u,v) + O(\epsilon^2),$$ 
the function $w $ can be rewritten as 
\begin{equation}\label{13}
 w= -u + \frac{1}{n+1} u^{n+1}  + \epsilon g_2(u,v) + O(\epsilon^2).
\end{equation}
By substituting  (\ref{13}) into the third equation of system  (\ref{7}), we have
\begin{equation}\label{14}
 \sqrt{c} (-1+u^n) \epsilon v  + O(\epsilon^2) = - \epsilon g_1(u,v) - \frac{\epsilon}{\sqrt{c}} v + O(\epsilon^2),
\end{equation}
which derives 
\begin{equation*}
g_2(u,v) =  - \sqrt{c} \left( u^n v + (-1  + \frac{1}{c}) v    \right).
\end{equation*}
Thus, the system (\ref{7}) on $M_\epsilon$ is given by
 \begin{equation}\label{15}
 \left\{\begin{split}
&\dfrac{\mathrm{d}u}{\mathrm{d} \tau}=v,\\
&\dfrac{\mathrm{d}v}{\mathrm{d} \tau}=-u + \frac{1}{n+1} u^{n+1}  + \epsilon (-\sqrt{c}) 
\left( u^n  + (-1+\frac{1}{c})     \right) v + O(\epsilon^2).\\
\end{split}\right.
\end{equation}

  Next, we  analyze whether the periodic orbit will continue to exist for the above system.
Note that the unperturbed Hamiltonian system (\ref{8}) can be described by  the Hamiltonian function $H.$
Fix an initial data $(u_0,0)$, $0<u_0< \sqrt[n]{n+1}.$ Let $(u(\tau),v(\tau))$ be the solution of (\ref{15})
with $(u(0),v(0))=(u_0,0).$ Then there exist $\tau_1>0$ and  $\tau_2<0$ such that
$$v(\tau)>0  ~~ \mathrm{for} ~ \tau\in (0,\tau_1),~~~~v(\tau_1)=0$$
and
$$v(\tau)<0  ~~ \mathrm{for} ~ \tau\in (\tau_2,0),~~~~v(\tau_2)=0.$$
Define a function $\Psi$ by
\begin{equation}\label{z-1-7}
\Psi(u_0,c,\epsilon) =   \int_{\tau_2}^{\tau_1}   \dot{H} (u,v) \mathrm{d} \tau,
\end{equation}
where the dot $``."$ represents the derivative by $\tau,$ 
$H(u,v)$ is defined as (\ref{9}),
and
\begin{equation}\label{z-1-8}
 \dot{H} (u,v)=  \epsilon \sqrt{c} \left(   (1-u^n) v^2 -\frac{1}{c} v^2        \right) + \mathrm{O} (\epsilon^2).   
 \end{equation}
By  Newton-Leibniz formula, we have 
\begin{equation}\label{z-1-9}
\Psi(u_0,c,\epsilon) =   H (u(\tau_1),v(\tau_1)) -   H (u(\tau_2),v(\tau_2)). 
\end{equation}
Thus, $u(\tau)$ is a periodic solution of (\ref{15}) if and only if $\Psi(u_0,c,\epsilon)=0.$
Since $\Psi(u_0,c,0)=0,$ one can set $$\Psi(u_0,c,\epsilon)= \epsilon \hat{\Psi}(u_0,c,\epsilon).$$
By using (\ref{z-1-7}) and  (\ref{z-1-8}), we obtain 
\begin{equation}\label{z-w-1}
\begin{split}
 \lim_{\epsilon \rightarrow 0} \hat{\Psi}(u_0,c,\epsilon) &=   \sqrt{c} \int_{\tau_2}^{\tau_1} 
  \left(   (1-u_1^n) v_1^2 -\frac{1}{c} v_1^2        \right) \mathrm{d} \tau \\
  & := \hat{\Psi}(u_0,c), 
  \end{split}
\end{equation}
where $(u_1,v_1)$ is a solution of  system (\ref{8}) and this integral is 
performed on a level curve $H=H(u_0,0) \in   (0,d_n).$ 

Since
 \begin{equation}\label{z-1-10}
   \int_{\tau_2}^{\tau_1}   u_1^n v_1^2     \mathrm{d} \tau 
  =  \int_{\tau_2}^{\tau_1}   u_1^n {u_1^\prime}^2     \mathrm{d} \tau 
   =  \int_{\tau_2}^{\tau_1}   u_1^n {u_1^\prime}     \mathrm{d} u_1
     = - n \int_{\tau_2}^{\tau_1}   u_1^n {v_1}^2     \mathrm{d} \tau 
     -  \int_{\tau_2}^{\tau_1}   u_1^{n+1} u_1^{\prime\prime} \mathrm{d} \tau, 
\end{equation}
we have 
\begin{equation}\label{z-1-11}
   \int_{\tau_2}^{\tau_1}   u_1^n v_1^2     \mathrm{d} \tau 
     = - \frac{1}{n+1}   \int_{\tau_2}^{\tau_1}   u_1^{n+1} u_1^{\prime\prime} \mathrm{d} \tau. 
\end{equation}
When $n=0$, it yields  that
\begin{equation}\label{z-1-12}
   \int_{\tau_2}^{\tau_1}   v_1^2     \mathrm{d} \tau 
     = -   \int_{\tau_2}^{\tau_1}   u_1 u_1^{\prime\prime} \mathrm{d} \tau. 
\end{equation}
By (\ref{z-1-11}) and  (\ref{z-1-12}), we have
\begin{equation}\label{z-1-13}
   \int_{\tau_2}^{\tau_1}   (1-u_1^n)  v_1^2     \mathrm{d} \tau 
     =    \int_{\tau_2}^{\tau_1} ( - u_1 +  \frac{1}{n+1} u_1^{n+1} ) u_1^{\prime\prime} \mathrm{d} \tau
     =  \int_{\tau_2}^{\tau_1}  { u_1^{\prime\prime}  }^2   \mathrm{d} \tau. 
\end{equation}
Thus, we obtain 
\begin{equation}\label{z-w-2}
 \hat{\Psi}(u_0,c)=  
  \sqrt{c} \int_{\tau_2}^{\tau_1} 
  \left(  { u_1^{\prime\prime}  }^2  -\frac{1}{c} {u_1^\prime}^2        \right) \mathrm{d} \tau
  =  \frac{1 }{\sqrt{c}}  \left( c \int_{\tau_2}^{\tau_1}  { u_1^{\prime\prime}  }^2 \mathrm{d} \tau 
        - \int_{\tau_2}^{\tau_1}  {u_1^\prime}^2  \mathrm{d} \tau   \right),
\end{equation}
and the limit wave speed $c_0$ satisfies 
 \begin{equation}\label{z-w-3}
  c_0 \int_{\tau_2}^{\tau_1}  { u_1^{\prime\prime}  }^2 \mathrm{d} \tau 
        - \int_{\tau_2}^{\tau_1}  {u_1^\prime}^2  \mathrm{d} \tau  =0.
\end{equation}

\section{ Analysis by  Abelian integral and  involution    }\label{szeng3}

We first give some lemmas associated with algebraic operation,  Abelian integral and involution.

\begin{lemma}[\cite{Patra2024}]      \label{Le4}
Suppose that  $u$ and $\eta$ satisfy the algebraic equation
$$u^2 - \eta^2=\frac{2}{(n+1)(n+2)} (u^{n+2}-\eta^{n+2}),$$ 
where $n>0$ is an integer.
 Then 
  \begin{equation}\label{}
 \left( \sum_{k=1}^{n} k u^{n-k} \eta^{k-1}  \right) (  \frac{u^{n+1}}{n+1} -  u  ) 
   +  \left( \sum_{k=1}^{n} k \eta^{n-k} u^{k-1}  \right) ( \frac{\eta^{n+1}}{n+1}  -   \eta  )
   = \frac{n}{(n+1)(n+2)} \sum_{k=0}^{[\frac{n-1}{2}]} (u^{n-2 k} - \eta^{n-2 k}  )^2 (u \eta)^{2 k}.
   \end{equation}
\end{lemma}

\begin{lemma}[\cite{Patra2024,WeiM2025}]      \label{Le2}
Consider the system
 \begin{equation}\label{19}
 \left\{\begin{split}
&\dfrac{\mathrm{d}u}{\mathrm{d} \tau}=v,\\
&\dfrac{\mathrm{d}v}{\mathrm{d} \tau}=\epsilon   \left( b_0 +b_n q_n(u)  \right) v - L^\prime (u),
\end{split}\right.
\end{equation}
 where  $\epsilon>0$ is a small parameter, $b_0$ and $b_n$ are real numbers, $q_n(u)$ is a polynomial.
 
Assume that:

(1) The Hamiltonian function of the unperturbed system of  (\ref{19}) is in  the form of
 $H(u, v) = \frac{v^2}{2} + L(u)$ and satisfies $L^\prime (u) (u- u^*)>0$ (or <0)
for $ u\in(u_2,u^*) \bigcup (u^*,u_3) $, where $(u^*,0)$ is a center of the unperturbed system of (\ref{19}),
  $L(u_2) = L(u_3).$
 
(2)  Let $ \tilde{B}_n(h)=\frac{ \oint_{\Gamma_h} q_n(u) v  ~\mathrm{d} u}{ \oint_{\Gamma_h}  v   ~\mathrm{d} u}$ and
 $ T_n(u)=(n+1) \frac{ \int_{\eta (u)}^{u}  q_n(t)   ~\mathrm{d} t}{  \int_{\eta (u)}^{u}   ~\mathrm{d} t},$
where $\Gamma_h$ is a periodic annulus of the unperturbed system of  (\ref{19}), $h \in (h_2,h_3)$ and $\eta (u)$ is an 
involution defined on interval  $(u_2,u_3).$

Then $ T_n^\prime(u) >0 $ for  $u\in(u^*,u_3)$ implies that $ \tilde{B}_n^\prime(h) >0 $ for  $h \in (h_2,h_3).$

\end{lemma}

\begin{lemma}[\cite{Christopher2007}]      \label{Le3}
Suppose that the Abelian integral $I(h)$ of system (\ref{19})  satisfies   $I(h) \not\equiv 0$ for  $h \in (h_2,h_3)$.
If  $X_{H,\epsilon}$ has a limit cycle bifurcating from $\Gamma_{\tilde{h}}$, then $I(\tilde{h})=0$.
 If there exists $h_* \in (h_2,h_3)$ such that $I(h_*)=0$ 
 and $I^\prime (h_*) \neq 0$, then $X_{H,\epsilon}$
has a unique limit cycle bifurcating from  $\Gamma_{h_*}$.
\end{lemma}

 The linear perturbation  of system (\ref{15}) can be derived as
 \begin{equation}\label{16}
 \left\{\begin{split}
&\dfrac{\mathrm{d}u}{\mathrm{d} \tau}=v,\\
&\dfrac{\mathrm{d}v}{\mathrm{d} \tau}=-u + \frac{1}{n+1} u^{n+1}  + \epsilon (-\sqrt{c}) 
\left( u^n  + (-1+\frac{1}{c})     \right) v.\\
\end{split}\right.
\end{equation}
Obviously, the Abelian integral of
system  (\ref{16}) is
\begin{equation}\label{17}
\begin{split}
I(h)&=\oint_{\Gamma_h}  (-\sqrt{c}) 
\left( u^n  + (-1+\frac{1}{c})     \right) v ~\mathrm{d} u\\
& =    (-\sqrt{c})  (-1+\frac{1}{c})   \oint_{\Gamma_h}   v ~\mathrm{d} u  
 +   (-\sqrt{c})   \oint_{\Gamma_h}  u^n  v   ~\mathrm{d} u  \\
  & =    \sqrt{c}   B_0(h)  \left(       (1-\frac{1}{c}) - \frac{ B_n(h)}{B_0(h)}          \right), \\
\end{split}
\end{equation}
where 
\begin{equation}\label{bn-b0}
 B_n(h) = \oint_{\Gamma_h}  u^n  v~\mathrm{d} u,~~~~~~~~B_0(h) = \oint_{\Gamma_h}   v~\mathrm{d} u.
\end{equation}
Let  $$\tilde{B}_n(h)= \frac{B_n(h)}{B_0(h)}.$$
Then we  give two propositions, which will exhibit some properties of  $\tilde{B}_n(h).$

 \begin{proposition}\label{theorem-1}
The function $ \tilde{B}_n(h)$ satisfies   $ \tilde{B}_n^\prime(h) > 0 $ 
 for $ h \in (0,d_n) $ and $\lim\limits_{h \rightarrow 0} \tilde{B}_n(h)=0. $
\end{proposition}

  \begin{proof} 

 When $n$ is a positive even number, since $W(u)=\frac{  u^2  }{2} - \frac{ u^{n+2} }{(n+1)(n+2)},$  we have 
  $W^\prime(u)=u- \frac{ u^{n+1} }{n+1}=u(1-  \frac{ u^{n} }{n+1})<0$  for $-\sqrt[n]{n+1}<u<0,$ it follows that the function 
   $W(u): (-\sqrt[n]{n+1},0) \rightarrow (0,d_n)$ is
   strictly decreasing. Thus, there exists the inverse function 
     $\sigma_1:  (0,d_n) \rightarrow  (-\sqrt[n]{n+1},0)$  
     satisfies    
     $W(\sigma_1(h)) \equiv h$ for all $h \in  (0,d_n) .$
 Similarly,    
  when $0<u<\sqrt[n]{n+1},$ we obtain that
  $W^\prime(u)=u(1-  \frac{ u^{n} }{n+1})>0,$  the function 
   $W(u): (0,\sqrt[n]{n+1}) \rightarrow (0,d_n)$ is
   strictly increasing,  there exists the inverse function 
     $\sigma_2:  (0,d_n) \rightarrow  (0,\sqrt[n]{n+1})$  
     satisfies    
     $W(\sigma_2(h)) \equiv h$ for all $h \in  (0,d_n) .$
Since $W(-\sqrt[n]{n+1})=W(\sqrt[n]{n+1})$, $W(\sigma_1(h)) \equiv h \equiv  W(\sigma_2(h)) $
 for all $h \in  (0,d_n),$ we can construct an involution 
 $\eta(u)$
on the interval $(-\sqrt[n]{n+1},\sqrt[n]{n+1})$ as
 \begin{equation}\label{8-bu1}
 \eta(u)=\left\{\begin{split}
& ( \sigma_2 \circ W ) (u),~~~\mathrm{if}~ u\in(-\sqrt[n]{n+1},0);\\
& ( \sigma_1 \circ W ) (u),~~~\mathrm{if}~ u\in(0,\sqrt[n]{n+1});\\
& 0,~~~~~~~~~~~~~~~~~~~~\mathrm{if}~ u=0.\\
\end{split}\right.
\end{equation}
The graph of $\eta(u)$ is shown in Figure \ref{Fig-duihe}.
\begin{figure}[H]
\centering
\includegraphics[width=6in]{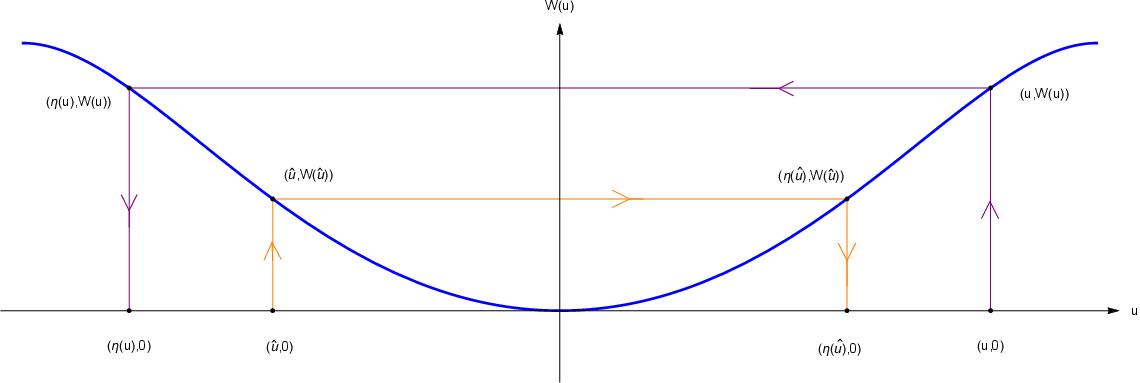}
\hspace{0.2in}
\caption{ The graph of image of $W$ and the involution $\eta$.  \label{Fig-duihe}}
\end{figure}

  When $n$ is an positive odd number,      
we obtain 
  $W^\prime(u)=u- \frac{ u^{n+1} }{n+1}=u(1-  \frac{ u^{n} }{n+1})>0$    for $0<u<\sqrt[n]{n+1},$  and the function 
   $W(u): (0,\sqrt[n]{n+1}) \rightarrow (0,d_n)$ is
   strictly increasing, and  there exists the inverse function 
     $\sigma_2:  (0,d_n) \rightarrow  (0,\sqrt[n]{n+1})$  
     satisfies    
     $W(\sigma_2(h)) \equiv h$ for all $h \in  (0,d_n) .$
Since $W(0)=0,$ 
  $\lim\limits_{u \rightarrow -\infty}  W(u) = +\infty$ and
  $W^\prime(u)=u(1-  \frac{ u^{n} }{n+1})<0$  for $u<0,$ 
   it follows that the function 
   $W(u): (-\infty,0) \rightarrow (0,+\infty)$ is
   strictly decreasing. Thus, there exists a constant $n_*<0$ with  $W(n_*)=W(\sqrt[n]{n+1})$ and 
     the inverse function 
     $\sigma_1:  (0,d_n) \rightarrow  (n_*,0)$  
    such that   
     $W(\sigma_1(h)) \equiv h$ for all $h \in  (0,d_n) .$
     Since $W(n_*)=W(\sqrt[n]{n+1})$, $W(\sigma_1(h)) \equiv h \equiv  W(\sigma_2(h)) $
 for all $h \in  (0,d_n),$ 
 we can construct an involution 
 $\eta(u)$
on the interval $(n_*,\sqrt[n]{n+1})$ as
 \begin{equation}\label{8-bu1}
 \eta(u)=\left\{\begin{split}
& ( \sigma_2 \circ W ) (u),~~~\mathrm{if}~ u\in(n_*,0);\\
& ( \sigma_1 \circ W ) (u),~~~\mathrm{if}~ u\in(0,\sqrt[n]{n+1});\\
& 0,~~~~~~~~~~~~~~~~~~~~\mathrm{if}~ u=0.\\
\end{split}\right.
\end{equation}

From the above we can see that for any integer $n>0$, the involution can be constructed.
 Without loss of generality, we  only discuss the case with $n$ is a positive even number.
   For the case with $n$ is a positive odd number,
  a similar discussion can be conducted, we omit it here.

Obviously, the involution  $\eta(u)$ satisfies the following   properties:\\
(i)~ $W(u)=W(\eta(u))$ for all $ u\in(-\sqrt[n]{n+1},\sqrt[n]{n+1}).$ That is
$$\frac{  u^2  }{2} - \frac{ u^{n+2} }{(n+1)(n+2)} =  \frac{  \eta^2  }{2} - \frac{ \eta^{n+2} }{(n+1)(n+2)},$$
which yields 
  $$  u^2 - \eta^2 =  \frac{ 2 }{(n+1)(n+2)} (  u^{n+2} - \eta^{n+2}  ),$$ where $\eta=\eta(u).$\\
 (ii) ~If $u\in(-\sqrt[n]{n+1},0),$ then $\eta(u)\in(0,\sqrt[n]{n+1}).$ \\ 
  (iii)~ If $u\in(0,\sqrt[n]{n+1}),$ then $\eta(u)\in(-\sqrt[n]{n+1},0).$ 
  
Taking 
$$ T_n(u)=(n+1) \frac{ \int_{\eta (u)}^{u}  t^n  ~\mathrm{d} t}{  \int_{\eta (u)}^{u}   ~\mathrm{d} t},$$
 by using Newton-Leibniz formula, it follows that $$ T_n(u)=\frac{u^{n+1} - \eta^{n+1}}{u- \eta}= \sum_{k=0}^{n} u^{n-k} \eta^k,$$
where $\eta=\eta(u).$

 By direct calculations and using the equation $\eta^\prime (u) = \frac{W^\prime(u)}{W^\prime(\eta )} $, we have 
     \begin{equation}\label{xs11}
      T_n^\prime (u)= \frac{\left( \sum_{k=1}^{n} k u^{n-k} \eta^{k-1}  \right) W^\prime(u)  
   +  \left( \sum_{k=1}^{n} k \eta^{n-k} u^{k-1}  \right) W^\prime(\eta)}{W^\prime(\eta)}.
   \end{equation}
 From Lemma \ref{Le4}, we have 
    \begin{equation}\label{xs12}
    \begin{split}
     \left( \sum_{k=1}^{n} k u^{n-k} \eta^{k-1}  \right) W^\prime(u)  
   +  \left( \sum_{k=1}^{n} k \eta^{n-k} u^{k-1}  \right) W^\prime(\eta)
   &= \left( \sum_{k=1}^{n} k u^{n-k} \eta^{k-1}  \right) ( u  - \frac{u^{n+1}}{n+1}   ) 
   +  \left( \sum_{k=1}^{n} k \eta^{n-k} u^{k-1}  \right) ( \eta  -   \frac{\eta^{n+1}}{n+1} )\\
& =  - \frac{n}{(n+1)(n+2)} \sum_{k=0}^{[\frac{n-1}{2}]} (u^{n-2 k} - \eta^{n-2 k}  )^2 (u \eta)^{2 k} <0. \\
   \end{split}
   \end{equation}
 Since  
  \begin{equation}\label{33}
 W^\prime(\eta)= \eta - \frac{\eta^{n+1}}{n+1}=\eta(1- \frac{\eta^n}{n+1})<0
\end{equation}
for   $-\sqrt[n]{n+1}<\eta(u)<0$, 
 by using Lemma \ref{Le4}, we have 
  $  T_n^\prime (u) > 0 $
for  $-\sqrt[n]{n+1}<\eta(u)<0< u< \sqrt[n]{n+1}$.
  From Lemma  \ref{Le2}, we get
   $ \tilde{B}_n^\prime(h) > 0 $ 
 for $ h \in (0,d_n). $
   The first part of the Proposition \ref{theorem-1}  is completed.

 When  $h \in (0,d_n),$ equation
 $\frac{  u^2  }{2} - \frac{ u^{n+2} }{(n+1)(n+2)}=h$  at least has two roots 
 $\pm\alpha(h)$  with  $\alpha(h)>0,$ and the two points
 $(\alpha(h),0)$ and $(-\alpha(h),0)$ 
 are the intersection points of the periodic orbit with the $u$-axis.
  Note that the orbit $(u(\tau),v(\tau))$ is on the level curve $H(u,v)=h$.


Since  
  \begin{equation}\label{bc3}
B_n(h)=   \oint_{\Gamma_h}  u^n  v   ~\mathrm{d} u =2  \int_{-\alpha(h)}^{\alpha(h)} 
 u^n \sqrt{-u^2+\frac{2}{(n+1)(n+2)}u^{n+2}+2h} ~ \mathrm{d} u := 2 J_n(h)
\end{equation}
and
   \begin{equation}\label{bc4}
\lim_{h \rightarrow 0}   \frac{J_n(h)}{J_0(h)}  =   \lim_{u \rightarrow 0} u^n =0,
\end{equation}
  we obtain that
  \begin{equation}\label{bc5}
   \lim_{h \rightarrow 0}   \tilde{B}_n(h) =  \lim_{h \rightarrow 0}   \frac{B_n(h)}{B_0(h)}   
     =   \lim_{h \rightarrow 0}   \frac{J_n(h)}{J_0(h)}  =0.
\end{equation}
  The second part of the Proposition \ref{theorem-1}  is completed. 
  
   \end{proof}

\begin{remark}
By  (\ref{bc3}), we have
 \begin{equation}\label{bc5-a1}
   \lim_{h \rightarrow d_n}   \tilde{B}_n(h) 
   =  \lim_{h \rightarrow d_n}   \frac{B_n(h)}{B_0(h)}   
     =   \lim_{h \rightarrow d_n}   \frac{J_n(h)}{J_0(h)}   
     = 
     \frac{\int_{0}^{\sqrt[n]{n+1}}  u^n \sqrt{-u^2+\frac{2}{(n+1)(n+2)}u^{n+2} + \frac{n(n+1)^{ \frac{2}{n} }   }{n+2}}
      ~ \mathrm{d} u }{   \int_{0}^{\sqrt[n]{n+1}}  \sqrt{-u^2+\frac{2}{(n+1)(n+2)}u^{n+2} + \frac{n(n+1)^{ \frac{2}{n} }   }{n+2}}
      ~ \mathrm{d} u     }.
  \end{equation}   
Since  the integral in equation   (\ref{bc5-a1}) cannot be calculated directly, 
  $ \lim\limits_{h \rightarrow d_n}  \tilde{B}_n(h)  $ is still an open problem.
However,  we can obtain the upper bound of  $ \tilde{B}_n(h) $  for some special positive even number $n$.
 For example, when $n=2,$
 \begin{equation}\label{bc5-a23}
   \lim_{h \rightarrow \frac{3}{4}  }   \tilde{B}_2(h) 
     =   \lim_{h \rightarrow \frac{3}{4}  }   \frac{J_2(h)}{J_0(h)}   
     = 
     \frac{\int_{0}^{\sqrt{3}}  u^2 \sqrt{-u^2+\frac{1}{6}u^{4} +  \frac{3}{2}    }
      ~ \mathrm{d} u }{  \int_{0}^{\sqrt{3}}  \sqrt{-u^2+\frac{1}{6}u^{4} +  \frac{3}{2}    }
      ~ \mathrm{d} u    }
      =  \frac{     \frac{3 \sqrt{2}}{5}   }{      \sqrt{2}    }=\frac{3}{5}.
  \end{equation}   

 When $n=4,$
 \begin{equation}\label{bc5-a23-1}
    \begin{split}
   \lim_{h \rightarrow \frac{\sqrt{5}}{3}  }   \tilde{B}_4(h) 
    & =   \lim_{h \rightarrow \frac{\sqrt{5}}{3}  }   \frac{J_4(h)}{J_0(h)}   
     = \frac{\int_{0}^{  \sqrt[4]{5}  }  u^4 \sqrt{-u^2+\frac{1}{15}u^{6} +  \frac{2\sqrt{5}}{3}    }
      ~ \mathrm{d} u }{  \int_{0}^{  \sqrt[4]{5}  }  \sqrt{-u^2+\frac{1}{15}u^{6} +  \frac{2\sqrt{5}}{3}    }
      ~ \mathrm{d} u    }\\
     &  =  \frac{ -  \frac{15}{16}  \sqrt{5} \left(2+\sqrt{3} \ln 2 - 2 \sqrt{3} \ln (1+\sqrt{3})\right)    }{   
       -\frac{1}{4} \sqrt{15} \ln (2-\sqrt{3})  }\\
      & =\frac{5 \left(2 \sqrt{3}+\ln (26-15 \sqrt{3})\right)}{  4 \ln (2-\sqrt{3})  }.
          \end{split}
  \end{equation}   
\end{remark}

 \begin{proposition}\label{theorem-2}
If  $ \tilde{B}_n(h)
\neq (1- \frac{1}{c})$
for any $ h \in (0,d_n),$ then there exists no limit cycle in system  (\ref{16}).
  If there exists a real number  $ \tilde{h} \in (0,d_n)$ such that  
   $ \tilde{B}_n(\tilde{h})
     = (1- \frac{1}{c}),$
 then there exists only one limit cycle in system  (\ref{16}).
 
\end{proposition}
\begin{proof} 

When $h \rightarrow 0$, the periodic annulus $\Gamma_h$ tend to the  point $(0,0)$. 
When $h \rightarrow  d_n$, the periodic annulus $\Gamma_h$ go to the heteroclinic orbit connecting
the two saddles  $(- \sqrt[n]{n+1},0)$ and $ (\sqrt[n]{n+1},0). $
Recall that the Abelian integral of
system  (\ref{16}) is
\begin{equation}\label{34}
I(h)=   \sqrt{c}   B_0(h)  \left(       (1-\frac{1}{c}) - \tilde{B}_n(h)        \right). \\
\end{equation}
By  Green formula, we have
\begin{equation}\label{18}
  B_0(h)= \oint_{\Gamma_h}   v ~\mathrm{d} u =\iint_{\hat{D}} \mathrm{d}u \mathrm{d}v=S_{\hat{D}}>0,
\end{equation}
where $S_{\hat{D}}$ represents the area of $\hat{D}$ which is the region enclosed by the curve $\Gamma_h$, 
 and the direction of curve  $\Gamma_h$ rotates clockwise. 
 If  $ \tilde{B}_n(h) \neq (1- \frac{1}{c})$
for any $ h \in (0,d_n),$ then $ I(h) \neq 0$
for any $ h \in (0,d_n).$ 
By using Lemma \ref{Le3}, we get the first part of Proposition \ref{theorem-2}.

By Proposition \ref{theorem-1}, we obtain that $\tilde{B}_n(h)$ is 
  monotonic  on  interval  $  (0,d_n).$
   If  $ \tilde{B}_n(\tilde{h}) = (1- \frac{1}{c}),$
 then $\tilde{h}$ is the unique root of $I(h)=0,$ and $I^\prime(\tilde{h}) \neq 0.$
By using Lemma \ref{Le3} again, we complete the second part of Proposition \ref{theorem-2}.
 \end{proof}


By using (\ref{34}), we obtain that when $\epsilon \rightarrow 0,$ there exists a periodic solution if
\begin{equation}\label{35}
 \tilde{B}_n(h) =  1-\frac{1}{c_0(h)}. 
\end{equation}
That is 
\begin{equation}\label{36}
 c_0(h)=  \frac{1}{1- \tilde{B}_n(h)}. 
\end{equation}
Since $ \tilde{B}_n^\prime(h) > 0 $  for $ h \in (0,d_n) $ 
and $\lim\limits_{h \rightarrow 0} \tilde{B}_n(h)=0, $
 we have 
 \begin{equation}\label{36-z1}
 c_0^\prime(h)=  \frac{ \tilde{B}_n^\prime(h)}{(1- \tilde{B}_n(h))^2}>0 
\end{equation}
 for $ h \in (0,d_n), $
 and  $\lim\limits_{h \rightarrow 0} c_0(h)=1. $
 
Using (\ref{z-w-2}), we have
 \begin{equation}\label{z-w-2-z1}
\frac{\partial  \hat{\Psi} }{\partial   c} (u_0(h),c_0,0)
  =  \frac{1 }{2\sqrt{c_0}}   \int_{\tau_2}^{\tau_1}  { u_1^{\prime\prime}  }^2 \mathrm{d} \tau 
        +   \frac{1 }{2 c_0  \sqrt{c_0}}   \int_{\tau_2}^{\tau_1}  {u_1^\prime}^2  \mathrm{d} \tau>0.
\end{equation}
  By implicit function theorem, we obtain that
  there exists exactly one smooth function $c_h(\epsilon)=c(\epsilon,h)$ for $ h \in (0,d_n) $
 and $\epsilon \in (0,\epsilon_n^*)$ such that
  \begin{equation}\label{z-w-2-z3}
 \hat{\Psi}  (u_0(h),c(\epsilon,h),\epsilon)=0~~~~~
 \mathrm{for} ~~ 0<h<d_n,  ~ 0<\epsilon<\epsilon_n^*.
\end{equation}
 We complete the proof of Theorem \ref{theorem-0}.


 Finally, we   get the upper bound of  $ c_0(h) $  for  $n=2$ and $n=4$.
 When $n=2,$ by (\ref{bc5-a23}) we have
 \begin{equation}\label{z-s-36}
  \lim_{h \rightarrow \frac{3}{4}  }    c_0(h)   =   \lim_{h \rightarrow \frac{3}{4}  }    \frac{1}{1- \tilde{B}_n(h)}= \frac{5}{2}, 
\end{equation}
which is agree with the result of Chen et al. \cite{ChenA2018}.
 When $n=4,$ by (\ref{bc5-a23-1})  we obtain 
 \begin{equation}\label{z-s-36-1}
  \lim_{h \rightarrow \frac{\sqrt{5}}{3}  }    c_0(h)   
  =   \lim_{h \frac{\sqrt{5}}{3} }    \frac{1}{1- \tilde{B}_4(h)}
  = \frac{1}{       1  -  \frac{5 \left(2 \sqrt{3}+\ln (26-15 \sqrt{3})\right)}{  4 \ln (2-\sqrt{3})  }        }.
\end{equation}
 The graph of the limit wave speed $c_0(h)$ for $n=4$   at $h\in(0,\frac{\sqrt{5}}{3})$
 is shown in Figure \ref{Fig03}.  The result shows that our theoretical analysis is consistent with the numerical simulation.
\begin{figure}[H]
\centering
\includegraphics[width=3in]{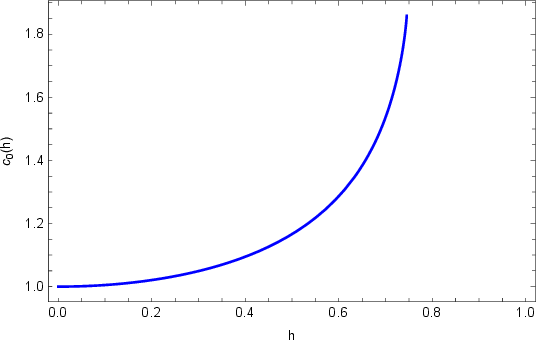}
\hspace{0.2in}
\caption{ The graph of the limit wave speed $c_0(h)$ for $n=4.$  \label{Fig03}}
\end{figure}

\noindent
{\bf Acknowledgements}
\vskip3mm
This work is supported by the Key  Project of  Education Department of Hunan Province (No.24A0667)
 and by the National Natural Science Foundation of China (No.12475049).

\subsection*{Ethical approval}
All authors agree to publish this article.

\subsection*{Competing interests}
The authors declare that they have no conflict of interests.

\subsection*{Authors' contributions}
All authors contributed equally to writing of this paper. All authors read and approved the final manuscript.

\subsection*{Availability of data and materials}
No data were used to support this study.

\end{document}